\documentclass[11pt]{amsart}
\usepackage{amssymb, latexsym, amsmath, amsfonts, amsrefs}

\newtheorem{thm}{Theorem}[section]
\newtheorem{cor}[thm]{Corollary}
\newtheorem{lem}[thm]{Lemma}
\newtheorem{prop}[thm]{Proposition}
\theoremstyle{definition}

\theoremstyle{remark}

\numberwithin{equation}{section}

\newcommand{\mf}{\mathbf}
\newcommand{\ra}{\rightarrow}
\newcommand{\z}{\zeta}
\newcommand{\pa}{\partial}
\newcommand{\ov}{\overline}

\newcommand{\ep}{\epsilon}
\newcommand{\no}{\noindent}

\newcommand{\la}{\lambda}

\newcommand{\La}{\Lambda}
\newcommand{\al}{\alpha}
\newcommand{\be}{\beta}
\newcommand{\ga}{\gamma}

\newcommand{\de}{\delta}
\newcommand{\De}{\Delta}

\title{Remarks on the metric induced by the Robin function III}
\author{Diganta Borah}
\address{Indian Institute of Science Education and Research, Pune-411008,
India}
\email{dborah@iiserpune.ac.in}
\thanks{The author was supported in part by the grant IFA-13 MA-21 from DST
under INSPIRE Faculty Award}

\begin{document}
\maketitle
\begin{abstract}
Let $D$ be a smoothly bounded pseudoconvex domain in $\mathbf C^n$, $n > 1$.
Using the Robin function $\La(p)$ that arises from the Green function $G(z,
p)$ for $D$ with pole at $p \in D$ associated with
the standard sum-of-squares Laplacian, N. Levenberg and H. Yamaguchi had
constructed a K\"{a}hler metric (the so-called
$\La$-metric) on $D$. In this article, we study the existence of geodesic
spirals for this metric.
\end{abstract}

\section{Introduction}

\noindent We continue the study of the metric induced by the Robin function on
strongly pseudoconvex domains in $\mf C^n$ from \cite{B} and \cite{BV}. To
quickly recall the setup, for 
a smoothly bounded pseudoconvex domain $D \subset \mf C^n$, the $\La$-metric on
D is defined as
\[
ds^2 = \sum_{\al, \beta = 1}^{n} \frac{\pa^2 \log(-\La)}{\pa z_{\al} \pa \ov
z_{\beta}} dz_{\al} \otimes d \ov z_{\beta}
\]
where $\La(p) = \lim_{z \ra p}(G(z,p) - \vert z-p \vert^{-2n+2})$ is the Robin
function associated to the $\mf R^{2n}$-Green function $G(z,p)$ with pole at $p
\in D$. It was proved in 
\cite{LY} that $\log(-\La)$ is strictly plurisubharmonic and hence $ds^2$
defines a K\"{a}hler metric, which is however not invariant under
biholomorphisms in general. Despite this 
seeming drawback, the $\La$-metric on a strongly pseudoconvex domain $D \subset
\mf C^n$ shares several properties with the Bergman metric (which is an
invariant K\"{a}hler metric!). 
For example, it was shown in \cite{BV} that the $\La$-metric on a strongly
pseudoconvex domain $D$ has the same boundary asymptotics as those of the
Bergman metric (and hence the 
Kobayashi and also the Carath\'{e}odory metric) which implies that it is
complete and that the metric space $(D, ds^2)$ is Gromov hyperbolic. Also, the
results of \cite{B} show that 
the holomorphic sectional curvature of $ds^2$ along normal directions approaches
$-1/(n-1)$ at the boundary, which is much like what is known for the Bergman
metric. To carry this 
similarity further, it was shown in \cite{Her} that on a nonsimply connected
strongly pseudoconvex domain $D$, every nontrivial homotopy class of closed
loops in $\pi_1(D)$ contains a 
closed geodesic in the Bergman metric. It is also known that (see \cite{Don},
\cite{DonFef}, \cite{Ohs89}) for a smooth strongly pseudoconvex domain $D$, the
space of harmonic 
forms $\mathcal H^{p, q}(D)$ with respect to the Bergman metric is zero
dimensional if $p + q \not= n$ while it is infinite dimensional for $p + q =
n$. 
Using the fact that the boundary asymptotics of the Bergman metrics match those
of the $\La$-metric, the exact analogues of both results were shown to hold for
the
$\La$-metric as well in \cite{B}. 

\medskip

The purpose of this note is to identify one more property that is shared by the
$\La$-metric and the Bergman metric thus increasing the list of their
similarities by one. We first need 
a definition. Let $(M, g)$ be a complete Riemannian manifold. A {\it geodesic
spiral} is a geodesic $c : \mathbb R \rightarrow M$ such that there is a compact
subset $K \subset M$ with 
$c(t) 
\in K$ for all $t \ge 0$ and $c$ is not closed. Further, if $c : \mf{R}
\rightarrow M$ is a non-constant geodesic and there exist times $t_1, t_2 \in
\mf{R}$ with $t_1 < t_2$ such 
that $c(t_1) = c(t_2)$, then the curve $c(t)$ restricted to the interval $[t_1,
t_2]$ will be called a {\it geodesic loop} through the point $c(t_1) = c(t_2)
\in M$. 

\begin{thm}
Let $D$ be a smoothly bounded strongly pseudoconvex domain in $\mf{C}^n$ and
suppose that
the universal cover of $D$ is infinitely sheeted. 
Then for each $p_0 \in D$ which does not lie on a closed geodesic there exists a
geodesic spiral for the $\La$-metric passing through $p_0$.
\end{thm}

The analogous result for the Bergman metric on smoothly bounded strongly
pseudoconvex domains can be found in \cite{Her}. The main step is Lemma 2.2 of
\cite{Her} which states that if 
$(M, g)$ is a complete Riemannian manifold whose universal cover is infinitely
sheeted and $x_0 \in M$ is a point through which no closed geodesic passes and
$K \subset M$ is a compact 
set which contains all possible geodesic loops through $x_0$, then there is a
geodesic spiral passing through $x_0$. By appealing to this, the theorem 
follows if we can show that there exists a compact set $K \subset D$ that
contains all the possible geodesic spirals through $p_0$. Thus the problem
reduces 
to finding such a compact $K$. To do this, let $\psi$ be a globally smooth
defining function for the strongly pseudoconvex domain $D$.

\begin{prop}\label{key}
There exists an $\epsilon = \epsilon(D) > 0$ such that for each geodesic 
$\gamma : \mf{R} \rightarrow D$ for the $\La$-metric with $\psi(\gamma(0)) >
-\epsilon$ and $(\psi \circ \gamma)^{\prime}(0) = 0$, it follows that
$(\psi \circ \gamma)^{\prime\prime}(0) > 0$.
\end{prop}

Take this $\ep > 0$ and let $2 \ep_1 = \min\{ \ep, \psi(p_0) \}$. Then
\[
K = \{ p \in D : \psi(p) \le -\ep_1\}
\]
is the compact set that we are seeking. Indeed, let $\gamma : [t_1, t_2]
\rightarrow D$ be a geodesic loop with $p_0 = \gamma(t_1) = \gamma(t_2)$.
Suppose that $\gamma$ does not lie in 
$K$, i.e., $\gamma$ enters the $\ep_1$ band around the boundary $\partial D$.
But
then, being a loop, it must turn back and hence $\psi \circ \gamma$ must have a
maximum somewhere, sat 
at $t_0 \in (t_1, t_2)$. This implies that $(\psi \circ \gamma)(t_0) >
-\epsilon, (\psi \circ \gamma)^{\prime}(t_0) = 0$ and $(\psi \circ
\gamma)^{\prime\prime}(t_0) < 0$
which contradicts the proposition. 
Thus it suffices to prove Proposition \ref{key}.

\medskip

\no {\it Acknowledgements :} The author would like to thank K.~Verma for the
suggestion of this problem and his encouragement and precious comments during
the
course of this work.

\section{Asymptotics of $\La$ and $\la$}
\noindent  We begin by strengthening some of the boundary asymptotics of the
$\La$-metric from \cite{BV}. 
Let $D$ be a $C^{\infty}$-smoothly bounded domain in $\mf{C}^n$ with
a $C^{\infty}$-smooth defining function $\psi$. In what follows, the standard
convention of denoting derivatives by suitable subscripts will be followed. For
example, $\psi_{\al}=\pa \psi/\pa p_{\al}$, $\psi_{\al\ov\be}=\pa^2
\psi/ \pa p_{\al\ov\be}$, etc. Also, let $\pa \psi=(\psi_1, \ldots, \psi_{n})$.
The \textit{normalised} Robin
function $\la$ associated to $(D, \psi)$ is defined by
\[
\la(p)=\begin{cases}
\La(p)\psi(p)^{2n-2} & \text{if $p \in D$},\\
-\vert \pa\psi(p)\vert^{2n-2} & \text{if $p \in \pa D$}.
\end{cases}
\]
This function has the following geometric significance: For $p \in D$, let
$D(p)$ be the domain in $\mf{C}^n$ obtained by applying the affine
transformation $z \mapsto (z-p)/(-\psi(p))$ to $D$, i.e.,
\[
D(p)= \left\{w \in \mf{C}^n: w= \frac{z-p}{-\psi(p)}\right\}.
\]
Observe that $D(p)$ contains the origin and by \cite{Y}*{Prop. 5.1},
\[
\La_{D(p)}(0) = \La(p) \psi(p)^{2n-2} = \la(p).
\]
Also, for $p \in \pa D$, let $D(p)$ be the half-space defined by
\[
D(p) = \Big\{w \in \mf{C}^n : 2 \Re \big(\sum_{\al=1}^n \psi_{\al}(p)
w_{\al}\big)-1<0\Big\}.
\]
Again, $D(p)$ contains the origin and by \cite{BV}*{(1.4)},
\[
\La_{D(p)}(0) = -\vert \pa \psi(p) \vert^{2n-2} = \la(p).
\]
Thus $\la(p)$ is the Robin constant for $D(p)$ at the origin. In \cite{LY}, this
geometric significance of $\la$ was used to understand its regularity near
the boundary $\pa D$. Indeed, let
\[
\mathcal{D}= \cup_{p \in D \cup \pa D} \big(p, D(p)\big) =\{(p,w): p \in D, w \in D(p)\}.
\]
The $\mathcal{D} : p \mapsto D(p)$ is a smooth variation of domains in
$\mf{C}^n$ defined by the smooth function on $\mf{C}^n \times \mf{C}^n$,
\begin{equation}\label{defn-f}
f(p,w)=2 \Re \left\{\sum_{\al=1}^n \int_{0}^1 (w_{\al}\psi_{\al}(p-\psi(p) t
w))dt\right\}-1.
\end{equation}
Suppose $g(p,w)$ is the Green function for $D(p)$ with pole at the origin. Then
we have the first variation formula
\begin{equation}\label{1st-v-f}
\frac{\pa g}{\pa p_{\al}}(p,w)=\frac{1}{2(n-1)\sigma_{2n}}\int_{\pa D(p)}
{k_1^{(\al)}(p, \z)} \vert \pa_{\z}g(p,\z) \vert \frac{\pa g_{w}}{\pa
n_{\z}}(p,\z) \, dS_{\z}, \quad p \in D, w \in D(p).
\end{equation}
Here, $\sigma_{2n}$ is the surface area of the unit sphere in $\mf{R}^{2n}$,
$dS_{\z}$ is the surface area element on $\pa D(p)$, $\pa_{\z}g=(\pa g/\pa
\z_1, \ldots, \pa g/\pa \z_n)$, $g_w(p,w)$ is the Green function for $D(p)$
with pole at $w$, $n_{\z}$ is the unit outward normal to $\pa D(p)$ at $\z$, and
\begin{equation}\label{k-1}
k_1^{(\al)}(p,\z)=\frac{\pa f}{\pa p_a}(p,\z)\Big/{\vert \pa _{\z} f(p,\z)
\vert}.
\end{equation}
When $p \in D$ converges to $p_0 \in \pa D$ and $w \in D(p)$ converges to $w_0
\in D(p_0)$, then the integral (\ref{1st-v-f}) converges to
\begin{equation}\label{1st-v-f-bdy}
\frac{1}{2(n-1)\sigma_{2n}}\int_{\pa D(p_0)} {k_1^{(\al)}(p_0, \z)} \vert
\pa_{\z}g(p_0,\z) \vert \frac{\pa g_{w_0}}{\pa n_{\z}}(p_0, \z) \, dS_{\z}.
\end{equation}
Then using a standard argument (see Step 6, Chapter 3 of \cite{LY}) it was shown that $\pa g/\pa p_{\al}(p_0, w_0)$ exists and is equal to the above integral. It follows that 
$g(p,w)$ is a $C^1$-smooth function of $p$ up to $\pa D$ and (\ref{1st-v-f}) holds for $p \in \pa D$ also. From \cite{Y}*{(1.3)}, $\la$ is also a $C^1$-smooth function of $p$ up to $\pa D$. Also, since
\[
\la_{\al}(p)=\frac{\pa g}{\pa p_{\al}}(p, 0), \quad p \in D, 
\]
we note that for all $p \in \ov D$,
\begin{equation}\label{1st-v-la}
\la_{\al}(p)=-\frac{1}{(n-1)\sigma_{2n}}\int_{\pa D(p)} {k_1^{(a)}(p, \z)} \vert \pa_{\z}g(p,\z) \vert^2 \, dS_{\z}.
\end{equation}
Similarly, using
the second variation
formula it was shown that $g(p,w)$ and thus $\la(p)$ is a $C^2$-smooth function
of $p$ up to $\pa D$.

In \cite{B}, we studied the boundary behaviour of $\La$ and $\la$ under a
$C^{\infty}$-perturbation of $D$. In this section, we derive some consequences
of these results that will be used to prove the main theorem. First, let
$D_{\nu}$, $\nu \geq 1$, be $C^{\infty}$-smoothly bounded domains in
$\mf{C}^n$ with $C^{\infty}$-smooth defining functions $\psi_{\nu}$, such that
$\{\psi_{\nu}\}$ converges in the $C^{\infty}$-topology on compact subsets of
$\mf{C}^n$
to $\psi$. The normalised Robin function associated to $(D_{\nu}, \psi_{\nu})$
will be denoted by $\la_{\nu}$. For multi-indices $A=(\al_1, \al_2, \ldots,
\al_n)$, $B=(\be_1, \be_2, \ldots, \be_n) \in \mf{N}$, let
\[
D^{A}=\frac{\pa^{\vert A \vert}}{\pa p_1^{\al_1}\pa p_2^{\al_2} \cdots \pa
p_n^{\al_n}} \quad \text{and} \quad D^{\ov B}=\frac{\pa^{\vert B \vert}}{\pa \ov
p_1^{\be_1}\pa \ov p_2^{\be_2}\cdots \pa \ov p_n^{\be_n}}
\]
and let $D^{A\ov B}=D^AD^B$. We have from \cite{B}:

\begin{thm}\label{asymp-La}
Suppose $p_{\nu} \in D_{\nu}$  
converges to $p_0 \in \pa D$. Define the half space
\[
\mathcal{H} = \Big\{ w \in \mathbf{C}^{n} : 2 \Re \Big( \sum_{\al=1}^{n}
\psi_{\al}(p_0) w_{\al} \Big) -1 < 0 \Big\}
\]
and let $\Lambda_{\mathcal{H}}$ denote the Robin function for $\mathcal{H}$.
Then for all multi-indices $A,B\in\mf{N}$,
\[
(-1)^{\vert A \vert + \vert B \vert} D^{A \ov B} \Lambda_{\nu} (p_{\nu}) \big(
\psi_{\nu}(p_{\nu})\big)^{2n-2+ \vert A \vert + \vert B \vert} \rightarrow D^{A
\ov B} \Lambda_{\mathcal{H}}(0)
\]
as $\nu \ra \infty$.
\end{thm}
For the half space $\mathcal{H}$, we have the explicit formula
\begin{equation}\label{G-H}
G_{\mathcal{H}}(p,z)=\vert z- p \vert^{-2n+2} - \vert z - p^*\vert^{-2n+2},
\end{equation}
where $p^*$ is the symmetric point of $p$ given by
\begin{equation}\label{sym-pt}
p^{*}=p-\left(\frac{2 \Re \big(\sum_{\al=1}^n \psi_{\al}(p_0)
p_{\al}\big)-1}{\vert
\pa \psi(p_0) \vert^2}\right) \ov\pa \psi(p_0).
\end{equation}
Therefore, the Robin function for $\mathcal{H}$ is
\[
\La_{\mathcal{H}}(p)= -\vert p - p^{*} \vert^{-2n+2} = - \vert \pa \psi(p_0)
\vert^{2n-2}\Big\{2
\Re\big(\sum_{\al=1}^n \psi_{\al}(p_0) p_{\al}\big)-1\Big\}^{-2n+2}.
\]
Thus we can compute $D^{A\ov B} \La_{\mathcal{H}}(0)$ explicitly for all
multi-indices $A,B$, and hence the above theorem provides the boundary
asymptotics of all derivatives of $\La_{\nu}$. For our record, we now write
down few of them in the corollary that follows. Let
\[
\mf{I}=\{1, \ldots, n\} \quad \text{and} \quad \ov{\mf{I}}=\{\ov 1, \ldots, \ov
n\}.
\]
If $a \in \ov{\mf{I}}$, then let $p_a=\ov p_{\ov a}$.
\begin{cor}\label{cor-asymp-La}
Under the hypothesis of Theorem \ref{asymp-La}, we have for all $a,b,c \in
\mf{I} \cup \ov{\mf{I}}$,
\begin{itemize}
\item [(a)] $\displaystyle \lim_{\nu \to \infty}
\La_{\nu}(p_{\nu})\big(\psi_{\nu}(p_{\nu})\big)^{2n-2}=-\big\vert \pa \psi(p_0)
\big\vert^{2n-2}$,
\item [(b)] $\displaystyle \lim_{\nu \to \infty}\frac{\pa \La_{\nu}}{\pa
p_{a}}(p_{\nu})\big(\psi_{\nu}(p_{\nu})\big)^{2n-1}=(2n-2)\psi_a(p_0)\big\vert
\pa
\psi(p_0)
\big\vert^{2n-2}$,
\item [(c)] $\displaystyle \lim_{\nu \to \infty} \frac{\pa^2 \La_{\nu}}{\pa
p_{a}\pa
p_{b}}(p_{\nu})\big(\psi_{\nu}(p_{\nu})\big)^{2n}
=-(2n-2)(2n-1)\psi_a(p_0)\psi_b(p_0)\big\vert \pa
\psi(p_0)
\big\vert^{2n-2}$, and
\item [(d)] $\displaystyle \lim_{\nu \to \infty} \frac{\pa^3 \La_{\nu}}{\pa
p_{a}\pa p_{b}\pa
p_{c}}(p_{\nu})\big(\psi_{\nu}(p_{\nu})\big)^{2n+1}=(2n-2)(2n-1)2n
\psi_a(p_0)\psi_b(p_0)\psi_c(p_0)\big\vert \pa \psi(p_0)
\big\vert^{2n-2}$.
\end{itemize}
\end{cor}
Now let $g_{\al \ov \be}$ and ${g_{\nu}}_{\al\ov \be}$, $1 \leq \al, \be \leq
n$, be the components of the $\La$-metric on $D$ and $D_{\nu}$ respectively.
Note that
\begin{equation}\label{g-al-be}
g_{\al\ov \be}=\frac{\pa^2  \log(-\La)}{\pa p_{\al} \pa \ov p_{\be}} =
\frac{\La_{\al \ov \be}}{\La}-\frac{\La_{\al}\La_{\ov \be}}{\La^2},
\end{equation}
and by differentiating this with respect to $p_{\ga}$, $1 \leq \ga \leq n$,
\begin{equation}\label{d-g-al-be}
\frac {\pa g_{\al \ov{\be}}} {\pa p_{\ga}} = \frac {\La_{\al \ov{\be} \ga}}
{\La} - \left(\frac {\La_{\al \ov{\be}}
\La_{\ga}} {\La^{2}} + \frac{\La_{\al \ga} \La_{\ov{\be}}} {\La^{2}} +
\frac{\La_{\ov{\be} \ga} \La_{\al}} {\La^{2}}
\right) + \frac{2 \La_{\al} \La_{\ov{\be}} \La_{\ga}} {\La^{3}}.
\end{equation}
Multiplying the corresponding equations for ${g_{\nu}}_{ \al \be}$ by
$\psi_{\nu}^2$ and $\psi_{\nu}^3$ respectively we obtain from Corollary
\ref{cor-asymp-La} that
\begin{cor}\label{asymp-g}
Under the hypothesis of Theorem \ref{asymp-La}, we have for all $\al,
\be, \ga \in \mf{I}$,
\begin{itemize}
\item [(a)] $\displaystyle \lim_{\nu \to \infty}{g_{\nu}}_{\al
\ov\be}(p_{\nu})\big(\psi_{\nu}(p_{\nu})\big)^2=(2n-2)\psi_{\al}(p_0)\psi_{\ov{
\be } } (p_0)$ , and
\item [(b)] $\displaystyle \lim_{\nu \to \infty} \frac{\pa {g_{\nu}}_{\al \ov
\be}}{\pa
p_{\gamma}}(p_{\nu})\big(\psi_{\nu}(p_{\nu})\big)^3=-2(2n-2)\psi_{\al}(p_0)\psi_
{ \ov { \be } }
(p_0)\psi_{\ga}(p_0)$.
\end{itemize}
\end{cor}

\noindent In the proof of the main theorem, we will be particularly interested
in these asymptotics  when $D$ and $D_{\nu}$, $\nu \geq 1$, are in the
following form:
\begin{align*}
(\dagger)
\begin{cases}
\text{$\bullet$ $D$ is strongly pseudoconvex, $0 \in \pa D$, and $\pa
\psi(0)=(0, \ldots, 0, 1)$},\\
\text{$\bullet$ $D_{\nu}$ is strongly pseudoconvex, $0 \in \pa D_{\nu}$, and
$\pa \psi_{\nu}(0)=(0, \ldots, 0,1)$, and}\\
\text{$\bullet$ $p_{\nu}$ lies on the inner normal to $\pa D_{\nu}$ and $p_{\nu}
\to 0$.}
\end{cases}
\end{align*}
Under this normalisation, observe in corollary \ref{cor-asymp-La} that if any of
the
derivatives is with respect to a variable other than $p_n$ or $\ov p_n$, then
the asymptotics become $0$. This means that these are not the
strongest asymptotics in this case. Similarly, the asymptotics in corollary
\ref{asymp-g} are
not the strongest one unless $\al=\be=\ga=n$. The problem with these weak
asymptotics  is that they make $\det(g_{\al \ov \be})$ indeterminate. Indeed,
\[
g_{\al\ov\be}(p)\sim \frac{\psi_{\al}(0)
\psi_{\ov{\be}}(0)}{\big(\psi(p)\big)^2},
\]
for $p$ close to $0$ and hence
\[
\det\begin{pmatrix}g_{\al \ov \be}(p)\end{pmatrix} \sim
\frac{\det\begin{pmatrix}\psi_{\al}(0)\psi_{\ov\be}(0)\end{pmatrix}}{
\big(\psi(p)\big)^ { 2n
}},
\]
which is apriori indeterminate since the numerator vanishes and $\psi(p) \to 0$
as $p \to 0$.
Thus it is necessary to improve these asymptotics for the calculation of
geodesics. This can be done for the first and the second order derivatives of
$\La_{\nu}$ and for ${g_{\nu}}_{\al\ov \be}$, by means of the following theorem
from \cite{B}:
\begin{thm}\label{asymp-la}
Suppose $p_{\nu} \in D_{\nu}$ converges to $p_0 \in \pa D$. Then for all $\al,
\be \in \mf{I}$,
\begin{itemize}
\item [(a)]  $\displaystyle\lim_{\nu \to \infty}\la_{\nu}(p_{\nu}) = \la (p_0)$,
\item [(b)] $\displaystyle\lim_{\nu \to \infty}\frac{\pa \la_{\nu}}{\pa
p_{\al}}(p_{\nu}) = \la_{\al}(p_0)$,
\item [(c)] $\displaystyle\lim_{\nu \to \infty}\frac{\pa^2 \la_{\nu}}{\pa
p_{\al}\pa \ov p_{\be}} (p_{\nu}) = \la_{\al\ov{\be}}(p_0)$.
\end{itemize}
\end{thm}
To see how this theorem leads to finer asymptotics, differentiate the normalised
Robin function
\[
\la=\La \psi^{2n-2},
\]
with respect to $a$, and then with respect to $b$, to obtain
\begin{equation}\label{La_a}
\La_a \psi^{2n-2} = \la_{a} - (2n-2) \la \psi^{-1} \psi_{a},
\end{equation}
and
\begin{equation}\label{La_ab}
\La_{ab} \psi^{2n-1} = \la_{ab}\psi - (2n-2) (\la_a\psi_{b}+\la_b \psi_a) +
(2n-2)(2n-1)\la \psi^{-1}\psi_a \psi_b - (2n-2)\la \psi_{ab}.
\end{equation}
While Theorem \ref{asymp-la} provides information about the derivatives of
$\la_{\nu}$ in the above formulae corresponding to $\La_{\nu}$, the
terms of the form $\psi_{\nu}^{-1} \big(\pa \psi_{\nu}/\pa p_a\big)$, can be
controlled by the following:
\begin{lem}\label{dpsi/psi}
Under the normalisation $(\dagger)$, we have for all $a \in \mf{I} \cup
\ov{\mf{I}}$, $a \neq n , \ov n$,
\[
\lim_{\nu \to \infty} \frac{1}{\psi_{\nu}(p_{\nu})}\frac{\pa \psi_{\nu}}{\pa
p_a}(p_{\nu}) = \frac{1}{2}\Big(\psi_{an}(0)+\psi_{a\ov n}(0)\Big).
\]
\end{lem}
For a proof, see \cite{B}*{Lemma 6.2}. We also need to compute $\la_a(0)$.
\begin{lem}\label{la_a-expt}
Under the normalisation $(\dagger)$, we have for all $a \in \mf{I} \cup
\ov{\mf{I}}$, $a \neq n , \ov n$,
\[
\la_a(0)=-(n-1)\big(\psi_{an}(0)+\psi_{a\ov n}(0)\big)
\]
\end{lem}
\begin{proof}
Let
\[
\mathcal{H}=D(0)=\{w \in \mf{C}^n : 2 \Re w_n-1<0\}.
\]
From (\ref{1st-v-la}) we have
\[
\la_a(0)=\frac{\pa g}{\pa p_a}(0,0)=-\frac{1}{(n-1)\sigma_{2n}}\int_{\pa
\mathcal{H}} {k_1^{(a)}(0, \z)} \vert \pa_{\z}g(0,\z) \vert^2 \, dS_{\z}.
\]
Let us first compute $k_1^{(a)}(0, \z)$ from (\ref{k-1}). Differentiating (\ref{defn-f}) with respect to $p_a$ and using $\psi_a(0)=0$, we obtain
\[
\frac{\pa f}{\pa p_a}(0, \z)=\sum_{j=1}^n \big(\z_{j} \psi_{aj}(0) + \ov \z_j \psi_{a\ov j}(0)
\big).
\]
Also,
\[
\frac{\pa f}{\pa \z_{\al}}(0, \z)= \psi_\al(0)
\]
so that $\vert \pa_{\z}f(0, \z) \vert =1$. Thus
\[
k_1^{(a)}(0,\z) = \frac{\pa f}{\pa p_a}(0,\z)\Big/{\vert \pa _{\z} f(0,\z)
\vert} = \sum_{j=1}^n \big(\z_{j} \psi_{aj}(0) + \ov \z_j \psi_{a\ov j}(0)
\big).
\]
From (\ref{G-H}),
\[
g(0,\z) = \vert \z \vert^{-2n+2} - \vert \z - 0^{*} \vert^{-2n+2}
\]
where $0^{*}=\big(0, \ldots, 0, 1\big)$
is the symmetric point of $0$ with respect to the hyperplane $\pa \mathcal{H}$.
Therefore,
\[
\frac{\pa g}{\pa \z_j}(0,\z)  = -(n-1) \big(\ov
\z_j \vert \z \vert^{-2n} -( \ov \z_j - \ov 0^{*}_j)\vert \z - 0^{*} \vert^{-2n}
\big),
\quad 1 \leq j \leq n.
\]
Note that for $\z \in \pa \mathcal{H}$, $\vert \z \vert = \vert \z -
0^{*}\vert$.
Therefore,
\[
\frac{\pa g}{\pa \z_i}(0,\z)  = -(n-1) \vert \z\vert^{-2n}\ov 0^{*}_j,
\quad \z \in \pa \mathcal{H}, 1 \leq j \leq n.
\]
This implies that
\[
\vert \pa_{\z}g(0, \z)\vert = (n-1) \vert \z \vert^{-2n}, \quad \z \in \pa H.
\]
Thus
\begin{equation}\label{la_a_0}
\la_a(0)= -\frac{(n-1)}{\sigma_{2n}}\sum_{j=1}^{n}\left\{\psi_{aj}(0) \int_{\pa
\mathcal{H}} \z_j \vert \z \vert^{-4n}dS_{\z}
+ \psi_{a\ov j}(0)\int_{\pa \mathcal{H}} \ov{\z}_j \vert \z \vert^{-4n}
dS_{\z}\right\}.
\end{equation}
Let us now compute the above integrals. First observe that for $1 \leq j \leq
n-1$,
\[
\frac{1}{\sigma_{2n}}\int_{\pa \mathcal{H}} \z_j \vert \z \vert^{-4n}
\, dS_{\z}=\frac{1}{\sigma_{2n}}\int_{-\infty}^{\infty}\cdots
\int_{-\infty}^{\infty}
\frac{x_j+iy_j}{(x_1^2+y_1^2+\cdots+1/4+y_n^2)^{2n}} \, dx_1dy_1\cdots
\widehat{dx_n}dy_n = 0,
\]
where as usual $\widehat{dx_n}$ means that the surface measure $dS_{\z}$ does
not contain the covector $dx_n$. Also,
\begin{multline*}
\frac{1}{\sigma_{2n}}\int_{\pa \mathcal{H}} \z_n \vert \z \vert^{-4n}
\, dS_{\z}=\int_{-\infty}^{\infty}\cdots \int_{-\infty}^{\infty}
\frac{1/2+iy_n}{(x_1^2+y_1^2+\cdots+1/4+y_n^2)^{2n}} \, dx_1dy_1\cdots
\widehat{dx_n}dy_n\\
=\frac{1}{2\sigma_{2n}}\int_{-\infty}^{\infty}\cdots \int_{-\infty}^{\infty}
\frac{1}{(x_1^2+y_1^2+\cdots+1/4+y_n^2)^{2n}} \, dx_1dy_1\cdots
\widehat{dx_n}dy_n\equiv \frac{1}{2}X.
\end{multline*}
Using polar coordinates,
\[
X=\frac{\sigma_{2n-1}}{\sigma_{2n}}\int_{0}^{\infty}\frac{r^{2n-2}}{(r^2
+1/4)^{2n}} \, dr \equiv I(2n-2,2n).
\]
Repeated integration by parts yields
\[
I(2n-2,2n)=\frac{2n-3}{2(2n-1)}\frac{2n-5}{2(2n-2)}\cdots
\frac{1}{2(n+1)}I(0,n+1).
\]
By the residue theorem,
\[
I(0,n+1)=\frac{\pi}{n!}(n+1)(n+2)\ldots(2n).
\]
Also, since $\sigma_{m}=2\pi^{m/2}/\Gamma(m/2)$,
\[
\frac{\sigma_{2n-1}}{\sigma_{2n}}=\frac{1}{\sqrt{\pi}}\frac{\Gamma(n)}{
\Gamma(n-1/2)}=\frac{1}{\pi} \frac{2^{n-1} (n-1)!}{(2n-3)(2n-5)\ldots 1}.
\]
Thus,
\begin{multline*}
X=\left\{\frac{1}{\pi} \frac{2^{n-1} (n-1)!}{(2n-3)(2n-5)\ldots
1}\right\} \left\{ \frac{2n-3}{2(2n-1)}\frac{2n-5}{2(2n-2)}\cdots
\frac{1}{2(n+1)} \right\} \left\{\frac{\pi}{n!}(n+1)(n+2)\ldots(2n)\right\}=2,
\end{multline*}
and hence
\[
\frac{1}{\sigma_{2n}}\int_{\pa \mathcal{H}} \z_n \vert \z \vert^{-4n} dS_{\z}=1.
\]
It follows from (\ref{la_a_0}) that
\[
\la_a(0)=-(n-1)\big(\psi_{an}(0)+\psi_{a\ov n}(0)\big).
\]
\end{proof}

\begin{cor}\label{f-asymp-La}
Under the normalisation $(\dagger)$, we have for all $a,b \in \mf{I} \cup
\ov{\mf{I}}$, $a \neq n , \ov n$,
\begin{itemize}
\item [(a)] $\displaystyle \lim_{\nu \to \infty} \frac{\pa \La_{\nu}}{\pa
p_a}(p_{\nu})\big(\psi_{\nu}(p_{\nu})\big)^{2n-2} = 0$,
\item [(b)] $\displaystyle \lim_{\nu \to \infty} \frac{\pa^2 \La_{\nu}}{\pa p_a
\pa p_b}(p_{\nu})\big(\psi_{\nu}(p_{\nu})\big)^{2n-1} =
-(n-1)\psi_b(0)\big(\psi_{an}(0)+\psi_{ a\ov n}(0)\big)
+(2n-2)\psi_{ab}(0)$.
\end{itemize}
\end{cor}
\begin{proof}
Applying Theorem \ref{asymp-la} and Lemma \ref{dpsi/psi} to (\ref{La_a})
corresponding to $\La_{\nu}$,
\[
\lim_{\nu\to \infty}\frac{\pa \La_{\nu}}{\pa p_a}(p_{\nu})\psi(p_{\nu})^{2n-2}
= \la_a(0)-(n-1)\la(0)\big(\psi_{an}(0)+\psi_{a\ov n}(0)\big) = 0,
\]
in view of Lemma \ref{la_a-expt} and the fact that $\la(0)=-\vert \pa
\psi(0) \vert^{2n-2}=-1$. Hence (a) is proved.

Applying Theorem \ref{asymp-la} and Lemma \ref{dpsi/psi} to (\ref{La_ab})
corresponding to $\La_{\nu}$, we obtain
\begin{multline*}
\lim_{\nu \to \infty} \frac{\pa^2 \La_{\nu}}{\pa p_a
\pa p_b}(p_{\nu})\psi_{\nu}(p_{\nu})^{2n-1} =
-(2n-2)\la_a(0)\psi_b(0)+(n-1)(2n-1)\la(0)\psi_b(0)\big\{\psi_{an}(0)+\psi_{
a\ov n}(0)\big\}\\
-(2n-2)\la(0)\psi_{ab}(0)=-(n-1)\psi_b(0)\big\{\psi_{an}(0)+\psi_{ a\ov
n}(0)\big\}
+(2n-2)\psi_{ab}(0),
\end{multline*}
in view of Lemma \ref{la_a-expt} and the fact that $\la(0)=-1$. Hence (b) is
proved.
\end{proof}

Now, multiplying (\ref{g-al-be}) by $\psi$, we may write
\[
g_{\al \ov \be} \psi = \frac{\La_{\al \ov \be} \psi^{2n-1}}{\La \psi^{2n-2}} -
\frac{(\La_\al \psi^{2n-2})(\La_{\ov \be} \psi^{2n-1})}{(\La \psi^{2n-2})^2}.
\]
Applying Corollary \ref{f-asymp-La} to the above formula corresponding to
$g_{\nu}$, we obtain the following:
\begin{cor}\label{f-asymp-g}
Under the normalisation $(\dagger)$, we have for all $\al,\be \in \mf{I}$, $\al
\ne n$,
\[
{g_\nu}_{\al\ov \be}(p_{\nu})\psi_{\nu}(p_{\nu})=
(n-1)\psi_{\ov \be}(0)\big(\psi_{\al n}(0)+\psi_{ \al \ov
n}(0)\big)-(2n-2)\psi_{\al \ov \be}(0).
\]
\end{cor}
These asymptotics are strong enough to controll $\det (g_{\al \ov
\be})$. Indeed,
\begin{cor}\label{asymp-det}
Under the normalisation $(\dagger)$, we have
\begin{equation}\label{lim-det}
\lim_{\nu \to \infty} \det\big({g_{\nu}}_{\al\ov\be}(p_{\nu})\big)
\psi_{\nu}(p_{\nu})^{n+1} =
(-1)^{n-1}(2n-2)^{n}\det\begin{pmatrix}\psi_{\al \ov \be}(0)\end{pmatrix}_{1\leq
\al, \be \leq n-1},
\end{equation}
which is nonzero as $D$ is strongly pseudoconvex.
\end{cor}

\begin{proof}
Let $(\Delta_{\al \ov \be})$ be the cofactor
matrix of $(g_{\al\ov\be})$. Then
\begin{equation}\label{det}
\det(g_{\al\ov\be})=g_{n\ov 1} \De_{n \ov 1}+\ldots+g_{n\ov n} \De_{n\ov n}.
\end{equation}
Note that
\begin{multline*}
\De_{n \ov \be}\psi^{n-1} = \psi^{n-1}(-1)^{n+\be} \det
\begin{pmatrix}
g_{1\ov 1} & \ldots  & g_{1\ov{\be-1}}  & g_{1 \ov{\be+1}} & \ldots & g_{1\ov
n}\\
\vdots & \vdots & \vdots & \vdots & \vdots\\
g_{n-1\ov 1} & \ldots & g_{n-1\ov{\be-1}} &g_{n-1\ov{\be+1}} & \ldots & g_{n-1
\ov n}
\end{pmatrix}\\
= (-1)^{n+\be}\det\begin{pmatrix}
g_{1\ov 1}\psi & \ldots  & g_{1\ov{\be-1}}\psi  & g_{1 \ov{\be+1}}\psi & \ldots
& g_{1\ov n}\psi\\
\vdots & \vdots & \vdots & \vdots & \vdots\\
g_{n-1\ov 1}\psi & \ldots & g_{n-1\ov{\be-1}}\psi &g_{n-1\ov{\be+1}}\psi &
\ldots & g_{n-1 \ov n}\psi
\end{pmatrix}.
\end{multline*}
Applying Corollary \ref{f-asymp-g} to the above formula corresponding to
$g_{\nu}$, we observe that
\[
\lim_{\nu\to\infty} {\De_{\nu}}_{n\ov\be}(p_{\nu})
\big(\psi_{\nu}(p_{\nu})\big)^{n-1}
\]
exists and is finite. In addition, using $\psi_{j}(0)=0$ for
$1 \leq j<n$,
\[
\lim_{\nu \to \infty} {\De_{\nu}}_{n\ov n} \psi^{n-1}=
(-1)^{n-1}(2n-2)^{n-1}\det\begin{pmatrix}\psi_{\al\ov
\be}(0)\end{pmatrix}_{1\leq \al, \be \leq n-1}.
\]
Multiplying (\ref{det}) by $\psi^{n+1}$, we may write
\[
\det(g_{\al\ov\be}) \psi^{n+1}= (g_{n\ov 1} \psi^2)( \De_{n \ov 1}
\psi^{n-1})+\ldots+(g_{n\ov n} \psi^2)( \De_{n\ov n} \psi^{n-1}).
\]
Now applying the above asymptotics of the cofactors and Corollary \ref{asymp-g} to this formula corresponding to $g_{\nu}$, we obtain (\ref{lim-det}).
\end{proof}

\begin{cor}\label{asymp-g-inv}
Under the normalisation $(\dagger)$,
\[
\lim_{\nu\to\infty}\frac{1}{\psi_{\nu}(p_{\nu})^2}g^{n \ov \be}(p_{\nu})
\]
exists and is finite for all $\be \in \mf{I}$, and in particular,
\[
\lim_{\nu\to\infty}\frac{1}{\psi_{\nu}(p_{\nu})^2}g^{n \ov n}(p_{\nu})=
\frac{1}{2n-2}.
\]
\end{cor}

\begin{proof}
Dividing
\[
g^{n\ov \be}=\frac{\Delta_{n \ov \be}}{\det(g_{i \ov j})}
\]
by $\psi^2$, we may write
\[
\frac{1}{\psi^2}g^{n\ov
\be}=\frac{1}{\det(g_{i\ov j})\psi^{n+1}}\Delta_{n \ov \be}\psi^{n-1}.
\]
Now applying Corollary \ref{asymp-det} and the asymptotics of the cofactors in its proof
to the above formula corresponding to $g_{\nu}$ we obtain the desired results.
\end{proof}

It is not known whether $\la$ is $C^3$-smooth up to $\pa D$ and so the above procedure can not be applied to obtain finer asymptotics of the third order
derivatives of $\La$ and thus of the
derivatives of the metric components. However in \cite{LY}*{Chap. 6}, a relation
between the third order derivatives of $\La$ and certain derivatives of $g(p, w)$ was established
which can be used to obtain information about finer asymptotics. Indeed, recall
that
\[
g(p,w)=\big(\psi(p)\big)^{2n-2}G(p,z),
\]
where $p,z \in D$ and $w=(z-p)/\big(-\psi(p)\big)$. Differentiating the above
equation with respect to $z_{\al}$ and with respect to $p_{\al}$, away from the
diagonal $z=p$, we obtain
\[
\frac{1}{-\psi}\frac{\pa g}{\pa w_{\al}} = \psi^{2n-2}\frac{\pa
G}{\pa z_{\al}},
\]
and
\[
\frac{\pa g}{\pa p_{\al}}+\frac{1}{\psi}\frac{\pa g}{\pa
w_{\al}}+\frac{1}{\psi^2}\psi_{\al}
\sum_{i=1}^{n}\left\{(z_i-p_i)\frac{\pa g}{\pa w_i} + (\ov z_i-\ov p_i)
\frac{\pa g}{\pa \ov w_i} \right\} = (2n-2) \psi^{2n-3}\psi_{\al} G+
\psi^{2n-2}\frac{\pa G}{\pa p_{\al}}.
\]
Adding these two equations and using $w_i=(z_i-p_i)/(-(\psi(p))$,
$g=\psi^{2n-2}G$, we obtain
\begin{equation}\label{g_al-G_al}
\frac{\pa g}{\pa p_{\al}}- (2n-2)\frac{1}{\psi}\psi_{\al} g
-\frac{1}{\psi}\psi_{\al} \sum_{i=1}^{n}\left(w_i\frac{\pa
g}{\pa w_i} + \ov w_i \frac{\pa g}{\pa \ov w_i} \right) =
\psi^{2n-2}\left(\frac{\pa G}{\pa z_{\al}}+\frac{\pa G}{\pa p_{\al}}\right)
\end{equation}
away from the diagonal $z=p$.
Now set
\[
G_{\al}(p,z)=\left(\frac{\pa G}{\pa p_{\al}}+\frac{\pa G}{\pa
z_{\al}}\right)(p,z) \quad \text{for $(p,z) \in D \times D$, $\al=1, \ldots,
n$},
\]
which is, by \cite{LY}*{Prop. 6.1}, a real analytic, symmetric function in $D
\times D$, harmonic in $z$ and $p$ and satisfy
\begin{equation}\label{G_al-diag}
G_{\al}(p,p)=\frac{\pa \La}{\pa p_{\al}}(p).
\end{equation}
Also set
\begin{multline}\label{def-g_al}
g_{\al}(p,w) = \psi(p)\frac{\pa g}{\pa p_{\al}}(p,w)-\psi_{\al}(p)
\left\{(2n-2)g(p,w) + \sum_{i=1}^{n}\left(w_i\frac{\pa g}{\pa
w_i}(p,w) + \ov w_i \frac{\pa g}{\pa \ov w_i}(p,w) \right)\right\}\\
p \in\ov D, w \in D(p), 1\leq \al \leq n.
\end{multline}
which is, by \cite{LY}*{Prop. 6.2}, a harmonic function of $w \in D(p)$ for each
$p \in \ov D$, and satisfies
\begin{equation}\label{g_al-0}
g_{\al}(p,0)= \psi(p)\la_{\al}(p) - (2n-2) \psi_{\al}(p) \la(p) = \La_{\al}(p).
\end{equation}
Now (\ref{g_al-G_al}) can be written as
\[
g_{\al}(p,w)=\big(\psi(p)\big)^{2n-1}G_{\al}(p,z)
\]
for $p,z \in D$ and $w = (z-p)/(-\psi(p))$. Repeating the above calculation for
this relation, we obtain
\begin{equation}\label{g_albe-G_albe}
\frac{\pa g_{\al}}{\pa \ov p_{\be}}- (2n-1)\frac{1}{\psi} \psi_{\ov{\be}}
g_{\al} -\frac{1}{\psi}\psi_{\ov{\be}}
\sum_{i=1}^{n}\left(w_i\frac{\pa g_{\al}}{\pa w_i} + \ov w_i \frac{\pa
g_{\al}}{\pa \ov w_i} \right) =
\psi^{2n-1}\left(\frac{\pa G_{\al}}{\pa \ov z_{\be}}+\frac{\pa G_{\al}}{\pa
\ov p_{\be}}\right).
\end{equation}
Again set
\[
G_{\al\ov \be}(p,z) = \left(\frac{\pa G_{\al}}{\pa \ov p_{\be}}+\frac{\pa
G_{\al}}{\pa \ov z_{\be}}\right)(p,z) \quad \text{for $(p,z) \in D \times D$,
$1 \leq \al, \be \leq n$},
\]
and
\begin{multline}\label{def-g_albe}
g_{\al\ov\be}(p,w)= \psi(p)\frac{\pa g_{\al}}{\pa \ov p_{\be}}(p,w)-
(2n-1)\psi_{\ov{\be}}(p) g_{\al}(p,w)-\psi_{\ov{\be}}(p)
\sum_{i=1}^{n}\left(w_i\frac{\pa g_{\al}}{\pa w_i}(p,w) + \ov w_i \frac{\pa
g_{\al}}{\pa \ov w_i}(p,w) \right),\\ p \in \ov D, w \in D(p), 1 \leq \al,\be
\leq n.
\end{multline}
Then (\ref{g_albe-G_albe}) can be written as
\[
g_{\al\ov\be}(p,w)=\big(\psi(p)\big)^{2n}G_{\al\ov\be}(p,z),
\]
where $p,z \in D$ and $w=(z-p)/(-\psi(p))$. Differentiating the above relation
with respect to $z_c$, we obtain
\begin{equation}\label{g-G}
\frac{1}{-\psi}\frac{\pa g_{\al\ov \be}}{\pa w_{c}} = \psi^{2n}\frac{\pa
G_{\al\ov \be}}{\pa z_c}.
\end{equation}
On the otherhand, by \cite{LY}*{6.14},
\begin{equation}\label{La-G}
\frac{\pa^3 \La}{\pa p_{\al} \pa \ov p_{\be} \pa p_c}(p)=2\frac{\pa G_{\al\ov
\be}}{\pa z_c}(p,p).
\end{equation}
Combining (\ref{g-G}), (\ref{La-G}) with (\ref{def-g_albe}), we obtain
\begin{equation}\label{3rd-der-La}
\frac{\pa^3 \La}{\pa p_{\al} \pa \ov p_{\be} \pa p_c}(p) (\psi(p)\big)^{2n} =
-\frac{2}{\psi(p)}\frac{\pa g_{\al\ov \be}}{\pa w_{c}}(p,0)
=-2\frac{\pa^2 g_{\al}}{\pa w_c \pa \ov
p_{\be}}(p,0)+\frac{4n}{\psi(p)}\psi_{\ov{\be}}(p)
\frac{\pa g_{\al}}{\pa w_c}(p,0).
\end{equation}
Thus, information about the third order derivatives
of $\La$ can be obtained by studying the derivatives of $g_{\al}(p,w)$.

\begin{lem}\label{asymp-g_al}
Under the hypothesis of Theorem 2.4, we have for all $\al \in I$, $c \in I \cup
\ov I$,
\begin{align*}
\lim_{\nu \to \infty} \frac{\pa {g_{\nu}}_{\al}}{\pa w_c}(p_{\nu}, 0) =
\frac{\pa g_{\al}}{\pa w_c}(0,0)=(n-1)(2n-1)\psi_{\al}(p_0)\psi_{c}(p_0)
\vert \pa \psi(p_0) \vert^{2n-2},
\end{align*}
where ${g_{\nu}}_{\al}(p,w)$ is defined by (\ref{def-g_al}).
\end{lem}

\begin{proof}
We know that ${g_{\nu}}_{\al}(p_{\nu},w)$ is a harmonic function of $w \in
D_{\nu}(p_{\nu})$ and $g_{\al}(p_0, w)$ is a harmonic function of $w \in
\mathcal{H} = D(p_0)$. First, we will show that $\{{g_{\nu}}_{\al}(p_{\nu},w)\}$
converges uniformly on compact subsets of $\mathcal{H}$ to $g_{\al}(p_0, w)$.
The first equality then follows from the harmonicity of these
functions. Note that $\{D_{\nu}(p_{\nu})\}$ is a
$C^{\infty}$-perturbation of $\mathcal{H}$ (see the proof of Theorem 1.3 in
\cite{B}). Therefore by \cite{B}*{Prop 3.1},
$\{g_{\nu}(p_{\nu},w)\}$ converges
uniformly on compact subsets of $\mathcal{H}\setminus\{0\}$ to $g(p_0, w)$. By
harmonicity,
\[
\left\{\frac{\pa g_{\nu}}{\pa w_i}(p_{\nu},w)\right\}, \quad 1 \leq i \leq n,
\]
also converges uniformly on compact subsets of $\mathcal{H} \setminus\{0\}$ to
$(\pa g/\pa w_i)(p_0, w)$.  Also by \cite{B}*{Remark 4.5},
\[
\left\{\frac{\pa g_{\nu}}{\pa p_{\al}}(p_{\nu},w)\right\}
\]
converges uniformly on compact subsets of $\mathcal{H}$ to $(\pa g/\pa
p_{\al})(p_0,w)$. It follows from (\ref{def-g_al})
that $\{{g_{\nu}}_{\al}(p_{\nu},w)\}$ converges uniformly to $g_{\al}(p_0,w)$ on
compact subsets of $\mathcal{H} \setminus \{0\}$ and hence of $\mathcal{H}$ by
the mean value theorem.

Now to calculate $\frac{\pa g_{\al}}{\pa w_c}(p_0,0)$ explicitly, let us
write (\ref{def-g_al}) in the form
\begin{equation}\label{new-def-g_al}
g_{\al}(p,w)=\psi(p)\frac{\pa g}{\pa p_{\al}}(p,w) - (n-1) \psi_{\al}(p) \big(g_0(p,w) + \ov{g_0(p,w)}\big),
\end{equation}
where
\begin{equation}\label{def-g_0}
g_0(p,w) = g(p,w) + \frac{1}{n-1}\sum_{i=1}^{n}w_i \frac{\pa g}{\pa w_i}(p,w).
\end{equation}
Note that if $w_0\in \mf{C}^n$ then
for $w \neq w_0$,
\begin{equation}\label{der-formula}
\sum_{i=1}^{n} w_i \frac{\pa}{\pa w_i} \vert w - w_0 \vert^{-2n+2} = -(n-1)
\vert
w - w_0 \vert^{-2n} \sum_{i=1}^{n} w_i \ov{(w_i - w_{0i})}.
\end{equation}
This implies that the singularities in the right hand side of (\ref{def-g_0})
get cancelled and hence (\ref{def-g_0}) defines a harmonic function of $w \in
D(p)$ for each $p \in
\ov D$. From (\ref{G-H}), we have
\[
g(p_0,w) = \vert w \vert^{-2n+2} - \vert w - 0^{*} \vert ^{-2n+2},
\]
where $0^{*}=\ov\pa \psi(p_0)/\vert \pa \psi(p_0)\vert^2$ is the symmetric point
of the origin with respect to $\pa D(p_0)$.
Therefore,
\begin{multline}\label{g_0-expt}
g_0(p_0,w) = - \vert w - 0^{*} \vert^{-2n+2} + \vert w - 0^{*}
\vert^{-2n}\sum_{i=1}^{n} w_i \ov{(w_i- 0^{*}_i)}\\ =  - \vert w - 0^{*}
\vert^{-2n} \Big\{\vert w - 0^{*} \vert^2-\sum_{i=1}^n
w_i\ov{(w_i-0^{*}_i})\Big\}
= \vert w - 0^{*} \vert^{-2n} \sum_{i=1}^n 0^{*}_i \ov{(w_i-0^{*}_i})\\ = \vert
w - 0^{*} \vert^{-2n}\Big(-\vert 0^{*} \vert^2 + \sum_{i=1}^{n} 0^{*}_i \ov w_i
\Big).
\end{multline}
From this equation we obtain
\begin{equation}\label{der-g_0}
\frac{\pa}{\pa w_c}\big\{g_0(p_0,w) +\ov{ g_0(p_0,w)}\big\}\Big
\vert_{w=0}=-(2n-1)0^{*}_{\ov c} \vert 0^{*} \vert^{-2n} = -(2n-1) \psi_c(p_0)
\vert \pa \psi(p_0) \vert^{2n-2}.
\end{equation}
Finally,
\begin{multline}\label{g_al-expt}
g_{\al}(p_0,w)=  -(n-1)\psi_{\al}(p_0) \big\{g_0(p_0,w) +
\ov{g_0(p_0,w)}\big\}\\ =  (n-1)\psi_{\al}(p_0)\vert w - 0^{*}
\vert^{-2n}  \Big\{2\vert 0^{*} \vert^2 - \sum_{i=1}^{n} (0^{*}_i \ov w_i + \ov
0^{*}_i w_i)\Big\},
\end{multline}
and hence by (\ref{der-g_0}),
\begin{equation}\label{der-g_al-expt}
\frac{\pa g_{\al}}{\pa w_c}(p_0,0) = (n-1)(2n-1)\psi_{\al}(p_0)\psi_{c}(p_0)
\vert \pa \psi(p_0) \vert^{2n-2}
\end{equation}
as desired.
\end{proof}
Next we calculate the second the second order derivatives of $g_{\al}$. To
simplify the calculations, we will consider only a special case which is
required for the proof of them main theorem.
\begin{lem}\label{asymp-2nd-der-g_al}
Under the normalisat $(\dagger)$, we have for all $\al, \be \in I$, $\be \neq
n$ and for all $c \in I \cup \ov I$,
\begin{multline*}
\lim_{\nu \to \infty} \frac{\pa^2 {g_{\nu}}_{\al}}{\pa w_c \pa \ov
p_{\be}}(p_{\nu}, 0)=\frac{\pa^2 g_{\al}}{\pa w_c \ov \pa p_{\be}}(0,0)\\ = 
\begin{cases}
(n-1)(2n-1)\psi_{\al}(0)\psi_{\ov{\be}c}(0) & \text{if $c \neq n, \ov n$}\\
(n-1)(2n-1)\psi_{\al}(0)
\left\{\left(n+\frac{1}{2}\right)\psi_{\ov{\be} c}(0)+\left(n-\frac { 1
} { 2 } \right) \psi_{\ov{\be}\ov{c}}(0)\right\}\\ + (2n-1)
\psi_{\al\ov{\be}}(0) &
\text{if $c = n$ or $c= \ov n$}.
\end{cases}
\end{multline*}
\end{lem}
\begin{proof}
Let
\[
\mathcal{H}=D(0)=\big\{w \in \mf{C}^n : 2 \Re w_n - 1 < 0 \big\}.
\]
We will show that $\pa {g_{\nu}}_{\al}/\pa \ov p_{\be}$ converges uniformly on
compact subsets of $\mathcal{H}$ to $\pa g_{\al}/\pa \ov p_{\be}$.
Differentiating (\ref{def-g_al}) with respect to $\ov p_{\be}$,
\begin{multline}\label{der-g_al}
\frac{\pa g_{\al}}{\pa \ov p_{\be}} =\psi \frac{\pa^2 g}{\pa \ov p_{\be} \pa
p_{\al}}+ \frac{\pa \psi}{\pa \ov p_{\be}} \frac{\pa g}{\pa p_{\al}}  
- \frac{\pa \psi}{\pa p_{\al}} \left\{(2n-2) \frac{\pa g}{\pa \ov p_{\be}} +
\sum_{i=1}^n \Big(w_i \frac{\pa^2 g}{\pa \ov p_{\be} \pa w_i}+\ov w_i
\frac{\pa^2 g}{\pa \ov p_{\be} \pa \ov w_i}\Big) \right\}
\\-\frac{\pa^2 \psi}{\pa \ov p_{\be}\pa p_{\al}} \left\{ (2n-2) g
+\sum_{i=1}^{n} \Big(w_i \frac{\pa g}{\pa w_i}+\ov w_i \frac{\pa g}{\pa \ov
w_i}\Big) \right\}.
\end{multline}
By {Remark 4.5} of \cite{B}, $\{\frac{\pa g_{\nu}}{\pa p_{\al}}(p_{\nu},w)\}$
and similarly by the arguments of Section~5 of \cite{B}, $\{\frac{\pa^2
g_{\nu}}{\pa \ov p_{\be} \pa p_{\al}}(p_{\nu}, w)\}$ converges uniformly on
compact subsets of $\mathcal{H}$ to $\frac{\pa^2 g}{\pa \ov p_{\be} \pa
p_{\al}}(0, w)$. By
harmonicity $\{\frac{\pa^2 g}{\pa w_i \pa \ov p_{\be}}(p_{\nu},w)\}$ converges
uniformly on compact subsets of $\mathcal{H}$ to $\frac{\pa^2 g_{\nu}}{\pa w_i
\pa \ov p_{\be}}(0,w)$. As in the previous lemma, $\{g_{\nu}(p_{\nu},w)\}$ and
$\{\frac{\pa g_{\nu}}{\pa w_i}(p_{\nu},w)\}$ converges uniformly on compact
subsets of $\mathcal{H} \setminus \{0\}$ to $g(p_0,w)$ and $\frac{\pa g}{\pa
w_i}(0,w)$ respectively. Hence $\{\frac{\pa {g_{\nu}}_{\al}}{\pa \ov
p_{\be}}(p_{\nu},w)\}$ converges uniformly to $\frac{\pa g_{\al}}{\pa \ov
p_{\be}}(0,w)$ on compact subsets of $\mathcal{H} \setminus \{0\}$ and hence
of $\mathcal{H}$ by the mean value theorem.  The first
equality is now a consequence of harmonicity of these functions.

To calculate
\[
\frac{\pa^2 g_{\al}}{\pa w_c \pa \ov p_{\be}}(p_0,0),
\]
note that from (\ref{der-g_al}),
\begin{multline*}
\frac{\pa g_{\al}}{\pa \ov p_{\be}}(0,w) = - \psi_{\al}(0)
\left\{(2n-2) \frac{\pa g}{\pa \ov p_{\be}}(0,w) +
\sum_{i=1}^n \Big(w_i \frac{\pa^2 g}{\pa \ov p_{\be} \pa w_i}(0,w)+\ov w_i
\frac{\pa^2 g}{\pa \ov p_{\be} \pa \ov w_i}(0,w)\Big) \right\}\\
-\psi_{\al\ov\be}(0)\Big\{g_0(0,w)+\ov{g_0(0,w)}\Big\}.
\end{multline*}
Differentiating this with respect to $w_c$ and using (\ref{der-g_0}),
\begin{equation}\label{2nd-der-g_al}
\frac{\pa^2 g_{\al}}{\pa w_c \pa \ov p_{\be}}(0,0) = -(2n-1)\psi_{\al}(0)
\frac{\pa^2 g}{\pa w_c \pa \ov p_{\be}}(0,0)
+(2n-1)\psi_c(0)\psi_{\al\ov\be}(0).
\end{equation}
Now,

From the work in \cite{LY}*{Chapter 4},
\[
\frac{\pa g}{\pa \ov p_{\be}}(0,w) = \frac{1}{2(n-1)\sigma_{2n}}\int_{\pa
\mathcal{H}} \ov{k_1^{(\be)}(0, \z)} \vert \pa_{\z}g(0,\z) \vert \frac{\pa
g_{w}}{\pa n_\z} (0,\z)dS_{\z},
\]
where
\[
k_1^{(\be)}(0,\z) = \frac{\pa f}{\pa p_{\be}}(0,\z)\Big/{\vert \pa _{\z} f(0,\z)
\vert} = \sum_{j=1}^n \Big(\z_{j} \psi_{\be j}(0) +
\ov \z_{i}\psi_{\be \ov j}(0)  \Big),
\]
and $g_w(0, \z)$ is the Green function for $\mathcal{H}$ with pole at $w$.
From the explicit formula (\ref{G-H}),
\[
g_w(0,\z) = \vert \z - w \vert^{-2n+2} - \vert \z - w^{*} \vert^{-2n+2}, \quad
\z, w \in \mathcal{H},
\]
where
\[
w^{*}=\big(w_1, \ldots, w_{n-1}, (1-\Re w_n) + i \Im w_n\big)
\]
is the symmetric point of $w$ with respect to the hyperplane $\pa \mathcal{H}$.
Therefore,
\[
\frac{\pa g_w}{\pa \z_i}(0,\z)  = -(n-1) \big\{\vert \z - w \vert^{-2n} ( \ov
\z_i - \ov w_i) -\vert \z - w^{*} \vert^{-2n} ( \ov \z_i - \ov w^{*}_i)\big\},
\quad, \z, w \in \mathcal{H}, 1 \leq i \leq n.
\]
In particular, for $\z \in \pa H$, since $\vert \z - w \vert = \vert \z - w^{*}
\vert$,
\[
\frac{\pa g_w}{\pa \z_i}(0,\z)  = -(n-1) \vert \z-w \vert^{-2n}\ov w^{*}_i,
\quad w\in \mathcal{H},1 \leq i \leq n.
\]
This implies that
\[
\vert \pa_{\z}g(0, \z)\vert = (n-1) \vert \z \vert^{-2n}, \quad \z \in \pa H,
\]
and
\[
\frac{\pa g_w}{\pa n_{\z}}(0,\z)=\frac{\pa g_w}{\pa x_n}(0,\z)=-2(n-1)\vert \z -
w \vert^{-2n}(1-\Re w_n), \quad \z \in \pa H, w \in \mathcal{H}.
\]
Therefore,
\begin{multline*}
\frac{\pa g}{\pa \ov p_{\be}}(0,w)= -\frac{(n-1)(1-\Re
w_n)}{\sigma_{2n}}\sum_{j=1}^{n}\left\{\psi_{\ov\be j}(0) \int_{\pa \mathcal{H}}
\z_i
\vert \z \vert^{-2n} \vert \z - w \vert^{-2n}
dS_{\z}\right.\\
\left.+ \psi_{\ov\be \ov{j}}(0)
\int_{\pa \mathcal{H}} \ov{\z} _i \vert \z \vert^{-2n} \vert \z - w \vert^{-2n}
dS_{\z}\right\}.\end{multline*}
Differentiating with respect to $w_c$,
\begin{multline}\label{2nd-der-g-expt}
\frac{\pa^2 g}{\pa w_c \pa \ov p_{\be}}(0,0)=- \frac{n(n-1)}{\sigma_{2n}}
\sum_{j=1}^{n}\left\{\psi_{\ov{\be}j}(0) \int_{\pa \mathcal{H}}
\z_j \ov{\z}_c \vert \z \vert^{-4n-2} dS_{\z}+\psi_{\ov{\be}\ov{j}}(0) \int_{\pa
\mathcal{H}} \ov{\z}_j \ov{\z}_c\vert \z \vert^{-4n-2} dS_{\z}\right\}\\
+\frac{(n-1)}{\sigma_{2n}}\frac{\pa \Re w_n}{\pa w_c}
\sum_{j=1}^{n}\left\{\psi_{\ov{\be} j}(0) \int_{\pa \mathcal{H}}
\z_j \vert \z \vert^{-4n} \, dS_{\z}+ \psi_{\ov{\be} \ov{j}}(0) \int_{\pa
\mathcal{H}}
\ov{\z} _j \vert \z \vert^{-4n} \,
dS_{\z}\right\}
\end{multline}
We now consider two cases:

\medskip

Case I. $c \neq n, \ov n$. Let $1 \leq j \leq n$, and $1 \leq k \leq (n-1)$.
Then integrating with respect to $x_k$ and $y_k$ variables first,
\[
\int_{\pa \mathcal{H}}\z_j  \z_k \vert \z \vert^{-4n-2} dS_{\z}
=\int_{-\infty}^{\infty} \cdots \int_{-\infty}^{\infty}
\frac{(x_jx_k-y_jy_k)+i(x_jy_k+y_jx_k)}{(x_1^2 +y_1^2+ \cdots
+1/4+x_{2n}^2)^{2n+1}} \, dx_1dy_1\cdots \widehat{dx_{n}}dy_{n}  = 0,
\]
and also for $j \neq k$,
\[
\int_{\pa \mathcal{H}}\z_j \ov \z_k \vert \z \vert^{-4n-2} dS_{\z} =
\int_{-\infty}^{\infty} \cdots \int_{-\infty}^{\infty} \frac{(x_jx_k+y_jy_k) +
i(y_jx_k-x_jy_k)}{(x_1^2 +y_1^2+ \cdots +1/4+x_{2n}^2)^{2n+1}} \, dx_1dy_1\cdots
\widehat{dx_{n}}dy_{n} = 0.
\]
It follows from (\ref{2nd-der-g-expt}) that
\begin{equation}\label{2nd-der-g-C1}
\frac{\pa^2 g}{\pa w_c \pa \ov p_{\be}}(0,0)=
-\frac{n(n-1)}{\sigma_{2n}}\psi_{\ov{\be} c}(0)\int_{\pa
\mathcal{H}} \vert \z_c \vert^2 \vert \z \vert^{-4n-2} dS_{\z}.
\end{equation}
Now, if $1\leq k\leq n-1$,
\begin{multline}\label{def-A}
\frac{1}{\sigma_{2n}}\int_{\pa\mathcal{H}}\vert \z_k \vert^2 \vert \z
\vert^{-4n-2} dS_{\z} = \frac{1}{\sigma_{2n}}\int_{-\infty}^{\infty} \cdots
\int_{-\infty}^{\infty} \frac{x_k^2+y_k^2}{(x_1^2 +y_1^2+ \cdots
+1/4+x_{2n}^2)^{2n+1}} \, dx_1dy_1\cdots \widehat{dx_{n}}dy_{n}\\
=\frac{2}{\sigma_{2n}}\int_{-\infty}^{\infty} \cdots \int_{-\infty}^{\infty}
\frac{x_k^2}{(x_1^2 +y_1^2+ \cdots +1/4+y_{n}^2)^{2n+1}} \,
dx_1dy_1\cdots\widehat{dx_{n}}dy_{n} \equiv 2A.
\end{multline}
Note that for $K>0$ and $m>1$, integrating by parts,
\[
\int_{-\infty}^{\infty}\frac{t^2}{(t^2+K)^{m+1}} \, dt
=\frac{1}{2m}\int_{-\infty}^{\infty} \frac{1}{(t^2+K)^m} dt.
\]
Therefore, taking $m=2n$, $K=x_1^2+ y_1^2+ \cdots +\widehat{x_i^2}+y_i^2+\cdots
+1/4 + y_n^2$,
\begin{equation}\label{A}
A= \frac{1}{4n \sigma_{2n}}\int_{-\infty}^{\infty} \cdots
\int_{-\infty}^{\infty} \frac{1}{(x_1^2 +y_1^2+ \cdots + 1/4+y_n^2)^{2n}} \,
dx_1 dy_1\cdots \widehat{dx_{n}}dy_n=\frac{1}{4n}X = \frac{1}{2n}
\end{equation}
where $X$ is as in lemm \ref{la_a-expt}. Hence from
(\ref{2nd-der-g-C1}) and (\ref{def-A}),
\[
\frac{\pa^2 g}{\pa w_c \pa \ov p_{\be}}(0)= -(n-1)\psi_{\ov{\be}c}(0).
\]
Therefore, from (\ref{2nd-der-g_al}),
\[
\frac{\pa^2 g_{\al}}{\pa w_c \pa \ov p_{\be}}(0,0)=
(n-1)(2n-1)\psi_{\al}(0)\psi_{\ov{\be}c}(0) \quad \text{if $c \neq n, \ov n$}.
\]

Case II. $c=n$ or $c =\ov n$. Let $1 \leq j \leq n-1$. Then integrating with
respect to $x_j$ and $y_j$ variables first,
\[
\int_{\pa \mathcal{H}} \z_j \z_n \vert \z \vert^{-4n-2} dS_{\z} =
\int_{-\infty}^{\infty}
\cdots \int_{-\infty}^{\infty} \frac{(x_j/2-y_jy_n) + i(x_jy_n+y_j/2)}{(x_1^2
+y_1^2+ \cdots +1/4+x_{2n}^2)^{2n+1}} \, dx_1dy_1\cdots \widehat{dx_{n}}dy_{n} =
0
\]
and similarly
\[
\int_{\pa \mathcal{H}} \z_j \ov \z_n \vert \z \vert^{-4n-2} dS_{\z}=0, \quad
\int_{\pa \mathcal{H}}
\z_j \vert \z \vert^{-4n} dS_{\z}=0.
\]
It follows from (\ref{2nd-der-g-expt}) that
\begin{multline}\label{2nd-der-g-C2}
\frac{\pa^2g}{\pa w_c \pa \ov
p_{\be}}(0)=-\frac{n(n-1)}{\sigma_{2n}}\left\{\psi_{\ov{\be}c}(0) \int_{\pa
\mathcal{H}}
\vert \z_n \vert^2 \vert \z \vert^{-4n-2}
dS_{\z}+\psi_{\ov{\be}\ov{c}}(0) \int_{\pa \mathcal{H}} \z_{\ov
c}^2 \vert \z \vert^{-4n-2} dS_{\z}\right\}\\
+\frac{(n-1)}{2\sigma_{2n}}\left\{\psi_{\ov{\be}n}(0)
\int_{\pa \mathcal{H}} \z_n \vert \z \vert^{-4n} dS_{\z} +\psi_{\ov{\be}
\ov{n}}(0)
\int_{\pa \mathcal{H}} \ov \z_n \vert \z \vert^{-4n}
dS_{\z}\right\}.
\end{multline}
Now,
\begin{multline*}
\frac{1}{\sigma_{2n}}\int_{\pa \mathcal{H}}\vert \z_n \vert^2 \vert \z
\vert^{-4n-2} dS_{\z} = \frac{1}{4 \sigma_{2n}}\int_{-\infty}^{\infty} \cdots
\int_{-\infty}^{\infty} \frac{1}{(x_1^2 +y_1^2+ \cdots + 1/4+y_n^2)^{2n+1}}\,
dx_1 dy_1\cdots \widehat{dx_{n}}dy_n\\ +
\frac{1}{\sigma_{2n}}\int_{-\infty}^{\infty} \cdots \int_{-\infty}^{\infty}
\frac{y_n^2}{(x_1^2 +y_1^2+ \cdots +1/4+x_{2n}^2)^{2n+1}} \,
dx_1dy_1\cdots\widehat{dx_{n}}dy_{n} \equiv \frac{1}{4}B + A.
\end{multline*}
Then,
\begin{multline*}
B= \frac{1}{ \sigma_{2n}}\int_{-\infty}^{\infty} \cdots \int_{-\infty}^{\infty}
\frac{1}{(x_1^2 +y_1^2+ \cdots + 1/4+y_n^2)^{2n+1}}\, dx_1 dy_1\cdots
\widehat{dx_{n}}dy_n\\
=\frac{\sigma_{2n-1}}{\sigma_{2n}}\int_{0}^{\infty}\frac{r^{2n-2}}{(r^2+1/4)^{
2n+1}}dr = \frac{\sigma_{2n-1}}{\sigma_{2n}} I(2n-2, 2n+1).
\end{multline*}
As in Lemma \ref{la_a-expt},
\[
I(2n-2,2n+1)=\frac{2n-3}{2(2n)} \frac{2n-5}{2(2n-1)}\cdots
\frac{1}{2(n+2)}I(0,n+2),
\]
and
\[
I(0,n+2)=\frac{\pi}{(n+1)!}(n+2)(n+3)\cdots (2n+2),
\]
so that
\begin{multline*}
B= \left\{\frac{1}{\pi} \frac{2^{n-1} (n-1)!}{(2n-3)(2n-5)\ldots 1}\right\}
\left\{\frac{2n-3}{2(2n)} \frac{2n-5}{2(2n-1)}\cdots \frac{1}{2(n+2)}\right\}
\left\{\frac{\pi}{(n+1)!}(n+2)(n+3)\cdots (2n+2)\right\}\\=\frac{2(2n+1)}{n}.
\end{multline*}
Therefore,
\[
\frac{1}{\sigma_{2n}}\int_{\pa \mathcal{H}}\vert \z_n \vert^2 \vert \z
\vert^{-4n-2} dS_{\z} =\frac{(n+1)}{n}.
\]
Also,
\begin{multline*}
\frac{1}{\sigma_{2n}}\int_{\pa \mathcal{H}} \z_{\ov c}^2 \vert \z \vert^{-4n-2}
dS_{\z} =
\frac{1}{\sigma_{2n}}\int_{-\infty}^{\infty}\cdots \int_{-\infty}^{\infty}
\frac{(1/4-y_n^2)\pm iy_n}{(x_1^2 +y_1^2+ \cdots +1/4+x_{2n}^2)^{2n+1}} \,
dx_1dy_1\cdots\widehat{dx_{n}}dy_{n}\\ = \frac{1}{4}B-A=1,
\end{multline*}
and
\begin{multline*}
\frac{1}{\sigma_{2n}}\int_{\pa \mathcal{H}}\z_n \vert \z \vert^{-4n} dS_{\z}
=\frac{1}{\sigma_{2n}}\int_{-\infty}^{\infty}\cdots \int_{-\infty}^{\infty}
\frac{1/2+iy_n}{(x_1^2 +y_1^2+ \cdots +1/4+x_{2n}^2)^{2n}} \,
dx_1dy_1\cdots\widehat{dx_{n}}dy_{n}\\
=\frac{1}{2}\frac{1}{\sigma_{2n}}\int_{-\infty}^{\infty}\cdots
\int_{-\infty}^{\infty} \frac{1}{(x_1^2 +y_1^2+ \cdots +1/4+x_{2n}^2)^{2n}} \,
dx_1dy_1\cdots\widehat{dx_{n}}dy_{n} = 2nA=1.
\end{multline*}
Hence from (\ref{2nd-der-g-C2}),
\[
\frac{\pa^2 g}{\pa w_c \pa \ov p_{\be}}(0,0)
=-(n-1)\left\{\left(n+\frac{1}{2}\right)\psi_{\ov{\be} c}(0)+\left(n-\frac { 1
} { 2 } \right) \psi_{\ov{\be}\ov{c}}(0)\right\}.
\]
Thus from (\ref{2nd-der-g_al}),
\begin{multline*}
\frac{\pa^2 g_{\al}}{\pa w_c \pa \ov p_{\be}}(0,0)=
(n-1)(2n-1)\psi_{\al}(0)
\left\{\left(n+\frac{1}{2}\right)\psi_{\ov{\be} c}(0)+\left(n-\frac { 1
} { 2 } \right) \psi_{\ov{\be}\ov{c}}(0)\right\}\\ + (2n-1)
\psi_{\al\ov{\be}}(0), \quad \text{if $c = n$ or $c= \ov n$},
\end{multline*}
as desired.
\end{proof}
Combining the inormations about the derivatives of ${g_{\nu}}_{\al}$, we obtain
the following asymptotics of the third derivatives of $\La_{\nu}$:

\begin{prop}\label{asymp-3rd-der-La}
Under the normalisation $(\dagger)$, we have for all $\al,\be, \ga \in I$, $\be
\neq n$,
\begin{multline*}
\lim_{\nu \to \infty} \frac{\pa^3 \La_{\nu}}{\pa p_{\al} \pa \ov p_{\be} \pa
p_c}(p_{\nu})\psi_{\nu}(p_{\nu})^{2n}\\=
\begin{cases}
-2(n-1)(2n-1)\psi_{\al}(0)\psi_{\ov{\be}c}(0) &  \text{if $c \neq n, \ov n$},\\
-(n-1)(2n-1)\psi_{\al}(0) \big\{\psi_{\ov{\be}c}(0) - \psi_{\ov{\be}\ov{c}}(0)
\big\}-2(2n-1) \psi_{\al\ov{\be}}(0) & \text{if $c=n$ or
$c=\ov n$}.
\end{cases}
\end{multline*}
\end{prop}
\begin{proof}
Consider the formula (\ref{3rd-der-La}) corresponding to $\La_{\nu}$ and apply
Lemma \ref{dpsi/psi}, Lemma \ref{asymp-g_al} and Lemma \ref{asymp-2nd-der-g_al}
to obtain the desired result.
\end{proof}

We conclude this section with the following calculation:
\begin{cor}\label{fasymp-der-g}
Under the normalisation $(\dagger)$, we have for all $\al, \be \in I$, $\be
\neq n$ and $c \in I \cup \ov I$,
\[
\lim_{\nu \to \infty} \frac{\pa g_{\al \ov \be}}{\pa p_{c}}(p_{\nu})
\big(\psi(p_{\nu})\big)^{2}
\]
exists and is finite.
\end{cor}
\begin{proof}
From (\ref{d-g-al-be}),
\begin{multline*}
\frac{\pa g_{\al \ov \be}}{\pa p_{c}}\psi^2 = \frac{\La_{\al \ov \be c}
\psi^{2n}}{\La \psi^{2n-2}} - \frac{(\La_{\al\ov \be}\psi^{2n-1}) (\La_c
\psi^{2n-1})+(\La_{\al c} \psi^{2n})(\La_{\ov \be} \psi^{2n-2})+(\La_{\ov \be c}
\psi^{2n-1})(\La_{\al} \psi^{2n-1})}{(\La \psi^{2n-2})^2}\\
+\frac{2(\La_{\al}\psi^{2n-1})(\La_{\ov \be}
\psi^{2n-2})(\La_{c}\psi^{2n-1})}{(\La\psi^{2n-2})^3}.
\end{multline*}
Frist let $c \neq n, \ov n$. Then applying Corollay \ref{cor-asymp-La},
Corollary \ref{f-asymp-La} and Proposition \ref{asymp-3rd-der-La} to the above
formula corresponding to $\La_{\nu}$, we obtain
\[
\lim_{\nu\to \infty}\frac{\pa {g_{\nu}}_{\al\ov{\be}}}{\pa p_{c}}(p_{\nu})
\psi(p_{\nu})^2=2(n-1)(2n-1)\psi_{\al}(0)\psi_{\ov{\be}c}(0) 
-(2n-2)^2
\psi_{\al}(0)\psi_{\ov{\be}c}(0)=2(n-1)\psi_{\al}(0)\psi_{\ov{\be}c}(0).\]
Similarly for $c=n$ or $c=\ov n$,
\begin{multline*}
\lim_{\nu\to \infty}\frac{\pa {g_{\nu}}_{\al\ov{\be}}}{\pa p_{c}}(p_{\nu})
\psi(p_{\nu})^2=(n-1)(2n-1)\psi_{\al}(0) \big\{\psi_{\ov{\be}c}(0) -
\psi_{\ov{\be}\ov{c}}(0)
\big\}+2(2n-1)
\psi_{\al\ov{\be}}(0)\\+4(n-1)^2\psi_{\al}(0)\big\{\psi_{\ov{\be}n}(0)+\psi_{\ov
{\be}\ov{n}}(0)\big\}-4(n-1)^2\psi_{\al\ov{\be}}(0)-4(n-1)^2\psi_{\al}(0)\psi_{
\ov{\be}c}(0)\\
=(n-1)\psi_{\al}(0)\big\{(2n-1)\psi_{\ov{\be}c}(0) + (2n-3)
\psi_{\ov{\be}\ov{c}}(0)\big\}+2\{(2n-1)-2(n-1)^2\}\psi_{\al\ov{\be}}(0).
\end{multline*}
\end{proof}


\section{Geodesic spirals : Proof of Proposition \ref{key}}

\begin{proof}
We prove this proposition by contradiction. Suppose the assertion is not true.
Then there exists a sequence $\{c_{\nu}\}$ of
geodesics with the following properties:
\begin{itemize}
\item [(i)] There exists a point $a_0 \in \pa D$ such that $a_{\nu}:=
c_{\nu}(0)$
converges to $a_0$ as $\nu \to \infty$.

\item [(ii)] The unit vectors $u_{\nu}:=\frac{c^{\prime}_{\nu}(0)}{\vert
c^{\prime}_{\nu}(0)\vert}$
converges to a unit vector $u_0$.

\item [(iii)] We have $(\psi \circ c_{\nu})^{\prime}(0)=0$ and  $(\psi
\circ c_{\nu})^{\prime\prime}(0)\leq 0$ for each $\nu$.
\end{itemize}
Since the $\La$-metric is invariant under affine transformations, without loss
of generality let us assume that
\begin{itemize}
\item $a_0=0$, $\pa \psi(0)=(0,\ldots,0,1)$, and $v_0=(1,0, \ldots, 0)$.
\end{itemize}
If $\nu$ is sufficiently large, then the distance between $a_{\nu}$ and $\pa D$,
say $\de_{\nu}$, is realised by a unique point $\pi(a_{\nu}) \in \pa D$, i.e.,
\[
\de_{\nu} = d(a_{\nu}, \pa D) = \big\vert a_{\nu} - \pi(a_{\nu}) \big \vert.
\]
We again assume without loss of generality that this is true for all $\nu
\geq 1$. Now for each $\nu$, we apply a translation followed by
sufficiently many
rotations to transform the domain $D$ to a new domain $D_{\nu}$ with a global
defining function $\psi_{\nu}$, such that
\begin{itemize}
\item $\pi(a_{\nu}) \in \pa D$ corresponds to $0 \in \pa D_{\nu}$ and $\pa
\psi_{\nu}(0)=(0, \ldots, 0,1)$.

\item The geodesic $c_{\nu}$ in $D$ corresponds to the geodesic $\ga_{\nu}$ in
$D_{\nu}$ that has the following properties:
\begin{itemize}
\item [(a)] $p_{\nu}:=\ga_{\nu}(0) = (0, \ldots, 0,-\de_{\nu})$.
\item [(b)] $v_{\nu}:=\frac{\ga^{\prime}_{\nu}(0)}{\vert
\ga^{\prime}_{\nu}(0) \vert}=(1,0 \ldots,0)$.
\item [(c)] $(\psi_{\nu} \circ \ga_{\nu})^{\prime\prime}\leq 0$.
\end{itemize}
\end{itemize}
Note that the above three bullets impliy $D$ and $D_{\nu}$ are as in the
normalisation $(\dagger)$. In what follows we will derive a contradiction by
showing that
\[
\lim_{\nu \to \infty}\frac{(\psi_{\nu} \circ
\ga_{\nu})^{\prime\prime}(0)}{\vert \ga^{\prime}_{\nu} (0) \vert^2} >0.
\]
We start with the following lemma:
\begin{lem}\label{2nd-der-psi-ga}
If $\ga=(\ga_1, \ldots, \ga_n)$ is a geodesic of the $\La$-metric on $D$, then
\[
(\psi \circ \ga)^{\prime\prime}=-2 \Re \sum_{\al=1}^n\psi_{\al}(\ga)
\sum_{j,k=1}^{n}
\left(\sum_{\be=1}^{n} \frac{\pa g_{k \ov \be}}{\pa p_{j}}g^{\be \ov \al}
\right)(\ga) \, \ga^{\prime}_{j} \ga^{\prime}_{k} + 2 \Re \sum_{j,k=1}^n
\psi_{\al\be}(\ga) \ga^{\prime}_{\al} \ga^{\prime}_{\be} + 2
\mathcal{L}_{\psi}(\ga,
\ga^{\prime}).
\]
\end{lem}
\begin{proof}
Note that
\[
(\psi \circ \ga)^{\prime\prime}=2\Re \sum_{\al=1}^{n} \psi_{\al}(\ga)
\ga^{\prime\prime}_{\al} +2
\Re \sum_{\al,\be=1}^{n}
\psi_{\al\be}(\ga)\ga^{\prime}_{\al}\ga^{\prime}_{\be}+2\mathcal{L}_{\psi}
(\ga ,\ga^{\prime}).
\]
On the other hand, the equations of geodesic in the complexified form is given
by
\[
-\ga^{\prime\prime}_{\al}=\sum_{j,k=1}^{n} \left(\sum_{\be=1}^{n} \frac{\pa g_{k
\ov
\be}}{\pa p_{j}}g^{\be \ov \al} \right)(\ga) \,
\ga^{\prime}_{j}\ga^{\prime}_{k}.
\]
Substituting this in the above formula yields the lemma.
\end{proof}
Now, this lemma together with (b) implies that
\begin{multline}\label{2der-psi-ga}
\frac{(\psi_{\nu} \circ \ga_{\nu})^{\prime\prime}(0)}{\vert \dot \ga_{\nu}
(0)
\vert^2}= -2 \Re \left(\sum_{\al,
\be=1}^n
\frac{\pa \psi_{\nu}}{\pa p_\al} \frac{\pa g_{1 \ov \be}}{\pa p_1}g^{\be \ov
\al}\right)(p_{\nu})  + 2  \Re \frac{\pa^2\psi_{\nu}}{\pa p_1^2}
(p_{\nu}) + 2 \frac{\pa^2 \psi_{\nu}}{\pa p_1 \pa \ov p_1}(p_{\nu}) \\ \equiv -2
\Re I+ 2 \Re II + 2II.
\end{multline}

\noindent We will now compute the limit of $I$ as $\nu \to \infty$. For
convenience, we
will drop
the subscript $\nu$. We write $I$ as
\[
I = \sum_{\al=1}^{n-1} \sum_{\be=1}^n \psi_{\al} \frac{\pa g_{1 \ov \be}}{\pa
p_1}g^{\be \ov
\al} + \sum_{\be=1}^{n-1} \psi_n \frac{\pa g_{1 \ov \be}}{\pa p_1}g^{\be \ov n}
+\psi_n \frac{\pa g_{1 \ov n}}{\pa p_1}g^{n \ov n} \equiv A + B + C.
\]

Claim: $A\to 0$ as $\nu \to \infty$. Write
\[
A=\sum_{\al=1}^{n-1} \sum_{\be=1}^n \left(\frac{\psi_{\al}}{\psi}
\right)\left(\frac{\pa g_{1 \ov \be}}{\pa p_1}
\psi^3\right) \left(\frac{g^{\be \ov\al}}{\psi^2}\right).
\]
As $\nu\to \infty$, the first bracket converges by lemma \ref{dpsi/psi},
the second one converges to $0$ by corollary \ref{asymp-g} and third one
converges by corollary
\ref{asymp-g-inv}. Thus $A \to 0$ as $\nu \to \infty$.

\medskip

Claim: $B\to 0$ as $\nu \to \infty$. Write
\[
B= \sum_{\be=1}^{n-1} \psi_n \left(\frac{\pa g_{1 \ov \be}}{\pa p_1}
\psi^2\right)\frac{g^{\be \ov n}}{\psi^2}\]
As $\nu \to \infty$, the first bracket converges to $1$, the second one
converges to $0$ by corollary \ref{fasymp-der-g} and third one converges by
corollary
\ref{asymp-g-inv}. Thus $B\to 0$ as $\nu \to \infty$.

\medskip

Claim: $C \to \psi_{11}(0)$ as $\nu \to \infty$.
Write
\[
C=\psi_n \left(\frac{\pa g_{1\ov n}}{\pa p_1} \psi^2\right)
\left(\frac{g^{n\ov n}}{\psi^2}\right).
\]
As $\nu \to \infty$, $\frac{\pa \psi}{\pa p_n} \to 1$ and by Corollary
\ref{asymp-g-inv},
\[
\frac{g^{n\ov n}}{\psi^2} \to \frac{1}{2(n-1)}.
\]
Also by Corollary \ref{fasymp-der-g},
\[
\frac{\pa g_{1 \ov n}}{\pa p_1} = \ov {\frac{\pa g_{n \ov 1}}{\pa \ov p_1}} \to
\ov{2(n-1)\psi_{\ov{1} \ov{1}}(0)} = 2(n-1)\psi_{11}(0).
\]
Thus $C \to \psi_{11}(0)$ as $\nu \to \infty$.

It follows that $I\to \psi_{11}(0)$ as $\nu \to \infty$. Evidently $II\to
\psi_{11}(0)$ as $\nu \to \infty$. Hence from (\ref{2der-psi-ga}),
\[
\lim_{\nu \to \infty} \frac{(\psi_{\nu} \circ
\ga_{\nu})^{\prime\prime}(0)}{\vert \dot
\ga_{\nu}(0) \vert^2} = 2\psi_{1\ov1}(0)>0
\]
as $D$ is strongly pseudoconvex. This contradicts (c) and hence the
proposition is proved.
\end{proof}

\end{document}